\documentclass[reqno]{amsart}

\usepackage[T1]{fontenc}
\usepackage{graphicx}
\usepackage{enumitem}
\usepackage{amsmath,amsfonts,amssymb,amsthm,amscd,latexsym,cite}
\usepackage{mathrsfs}
\usepackage{xcolor}
\usepackage{hyperref}

\hypersetup{
  colorlinks=true,
  linkcolor=blue,
  citecolor=blue,
  urlcolor=blue
}

\newtheorem{thm}{Theorem}[section]
\newtheorem{theorem}[thm]{Theorem}

\newtheorem{lemma}[thm]{Lemma}

\newtheorem{remark}[thm]{Remark}

\newtheorem{problem}{Problem}

\newcommand{\cM}{{\mathcal M}}
\newcommand{\cN}{{\mathcal N}}

\newcommand{\cR}{{\mathcal R}}

\title
{Tingley's Problem for Haagerup Noncommutative $L_p$-Spaces, $1<p\ne 2<\infty$}

\author[Jinghao Huang]{J. Huang}
\thanks{Y. Zhu is the corresponding author. J. Huang was supported by the
  NNSF of China (Nos. 12031004, 12301160, 12471134 and 12671159).}
\address{Institute for Advanced Study in Mathematics of HIT, Harbin Institute
  of Technology, Harbin 150001, China}
\email{jinghao.huang@hit.edu.cn}

\author[Yunpeng Zhu]{Y. Zhu}
\address{Institute for Advanced Study in Mathematics of HIT, Harbin Institute
  of Technology, Harbin 150001, China}
\email{24s012030@stu.hit.edu.cn}

\begin{document}

\begin{abstract}
We solve Tingley's problem for Haagerup
noncommutative $L_p$-spaces, where $1<p\ne 2<\infty$, which gives an affirmative answer to Mori's problem \cite[Problem~6.3]{Mori}.
\end{abstract}

\maketitle

\section{Introduction}

\subsection{Background}

For a normed space $X$, write $S(X)=\{x\in X:\|x\|=1\}$. In 1987,
Tingley posed the following problem \cite{Tingley}.

\begin{problem}
Let $X$ and $Y$ be real normed spaces, and let
\(V_0:S(X)\to S(Y)\) be a surjective isometry. Does $V_0$ extend to a
surjective real-linear isometry from $X$ onto $Y$?
\end{problem}

Although the problem remains open in full generality, it has affirmative
solutions for many important classes, including commutative $\ell_p$-type and
$L_p$-spaces \cite{Ding2,Ding3,Ding2009,Tan1,Tan2,Tan3}, operator algebras and
operator spaces \cite{Tanaka,Polo3,Polo18,MO20,Mori}, the trace class
\cite{Polo2}, and the Schatten $p$-classes, for
$1<p<\infty$, $p\ne2$, in \cite{Polo}, and subsequently for all
$0<p<\infty$, $p\ne2$, in \cite{HuangZhu}.

\subsection{Tingley's problem for noncommutative
\texorpdfstring{$L_p$}{Lp}-spaces}

For a von Neumann algebra $\cM$, the Haagerup space $L_p(\cM)$ is the
noncommutative counterpart of a classical $L_p$-space. In 2018, Mori
formulated the following problem.

\begin{problem}
\cite[Problem~6.3]{Mori}\label{prob:Mori-noncommutative-Lp}
Let $1<p<\infty$, $p\ne2$, let $\cM$ and $\cN$ be von Neumann algebras,
and let
\[
  V_0:S(L_p(\cM))\longrightarrow S(L_p(\cN))
\]
be a surjective isometry between the unit spheres of their Haagerup
noncommutative $L_p$-spaces.
Does $V_0$ extend to a surjective real-linear isometry
from $L_p(\cM)$ onto $L_p(\cN)$?
\end{problem}

Mori settled the corresponding $L_1$-problem \cite{Mori}, and \cite{Polo}
settled the Schatten case. This paper gives a complete
affirmative answer to Problem \ref{prob:Mori-noncommutative-Lp} for arbitrary
von Neumann algebras.

\begin{theorem}\label{thm:main-result}
Let $\cM_1$ and $\cM_2$ be von Neumann algebras, and let
\(
  V_0:S(L_p(\cM_1))\longrightarrow S(L_p(\cM_2))
\)
be a surjective isometry, where $1<p\ne 2<\infty$. Then there exists a surjective real-linear isometry that extends \(V_0\). 
\end{theorem}
Notably, our argument requires neither separability nor semifiniteness and
therefore applies, in particular, to type~$\mathrm{III}$ von Neumann
algebras without invoking tracial arguments.

Following the strategy of \cite{Polo}, we first prove that $V_0$ preserves
orthogonality (i.e., 
 \(
 xy^*=0\) and \(x^*y=0\), defined by $x\perp y$) and is antipodal:
\[
  V_0(-x)=-V_0(x)
  \qquad \bigl(x\in S(L_p(\cM_1))\bigr).
\]
The strict convexity of noncommutative $L_p$-spaces, together with the Clarkson inequality established in
Lemma~\ref{lem:Clarkson}, yields orthogonal additivity with real
coefficients: whenever $x,y\in S(L_p(\cM_1))$ satisfy $x\perp y$,
\begin{equation}\label{eq:intro-orth-add}
  V_0(ax+by)=aV_0(x)+bV_0(y),
  \qquad
  a,b\in\mathbb R,\quad |a|^p+|b|^p=1.
\end{equation}

We next consider the radial extension
\[
  \widetilde V(0):=0,
  \qquad
  \widetilde V(x)
  :=\|x\|_pV_0\left(\frac{x}{\|x\|_p}\right)
  \quad (x\ne0).
\]
The definition of $\widetilde V$ and \eqref{eq:intro-orth-add} imply that $\widetilde V$ is real homogeneous and
additive on orthogonal elements. To prove its global real linearity, we
distinguish between the infinite- and finite-dimensional cases.

In the infinite-dimensional case, we first show that if there exists
$z\in S(L_p(\cM_1))$ such that $z\perp x,y$,
$\widetilde V$ preserves distances, that is
\begin{align}\label{V preserve}
  \|\widetilde V(x)-\widetilde V(y)\|_p=\|x-y\|_p.
\end{align}
For arbitrary $x,y$, it follows by the
bimodule continuity of Haagerup $L_p$-spaces \cite[Chapter~II]{Terp}, \cite[Section~2]{RX}  that \eqref{V preserve} holds.
Thus $\widetilde V$ is a surjective isometry from $L_p(\cM_1)$ to $L_p(\cM_2)$, and
the Mazur--Ulam theorem \cite{MazurUlam} implies that it is real linear.

In the finite-dimensional case, let $z_1,\ldots,z_m$ be the central atoms
of $\cM_1$. By \eqref{eq:intro-orth-add}, there exist corresponding central
atoms $w_1,\ldots,w_m$ of $\cM_2$ such that
\[
  \widetilde V\bigl(z_\alpha L_p(\cM_1)\bigr)
  =w_\alpha L_p(\cM_2),
  \qquad 1\leq\alpha\leq m.
\]
Accordingly,
\[
  L_p(\cM_1)
  =\bigoplus_{\alpha=1}^m z_\alpha L_p(\cM_1),
  \qquad
  L_p(\cM_2)
  =\bigoplus_{\alpha=1}^m w_\alpha L_p(\cM_2).
\]
Each central summand is a finite-dimensional Schatten $p$-class \cite[Section~1.14 and Remark~6.2.3]{DPS}. By
\cite{Polo}, for every $\alpha$, there is a surjective real-linear isometry
\[
  T_\alpha:
  z_\alpha L_p(\cM_1)\longrightarrow
  w_\alpha L_p(\cM_2)
\]
extending the corresponding restriction of $V_0$. Combining \eqref{eq:intro-orth-add},
\[
  \widetilde V=\bigoplus_{\alpha=1}^m T_\alpha,
\]
and $\widetilde V$ is a surjective real-linear isometry.

Combining the two cases, we conclude that the radial extension
$\widetilde V$ is a surjective real-linear isometry from
$L_p(\cM_1)$ onto $L_p(\cM_2)$ extending $V_0$. This proves
Theorem~\ref{thm:main-result}.

\section{Preliminaries}
In this section we recall the notation and facts concerning Haagerup
noncommutative $L_p$-spaces used in the proof of Theorem
\ref{thm:main-result}. Further details can be found in
\cite{HaagerupLp,Terp} and \cite[Section~1.2]{JungeXuReduction}.

Let $\cM$ be a von Neumann algebra, with no separability assumption, and
let $I$, $\mathcal Z(\cM)$, and $\mathcal P(\cM)$ denote its identity,
center, and projection lattice, respectively. For a normed space $X$, put
\[
  S(X):=\{x\in X:\|x\|=1\}.
\]

Fix a normal faithful semifinite weight $\varphi$ on $\cM$. Let
\[
  \cR_\varphi:=\cM\rtimes_{\sigma^\varphi}\mathbb R
\]
be the crossed product, let $(\theta_s)_{s\in\mathbb R}$ be its dual
action, and let $\tau_\varphi$ be the canonical normal faithful semifinite
trace on $\cR_\varphi$, characterized by
\[
  \tau_\varphi\circ\theta_s=e^{-s}\tau_\varphi
  \qquad(s\in\mathbb R).
\]
If $L_0(\cR_\varphi,\tau_\varphi)$ denotes the algebra of
$\tau_\varphi$-measurable operators affiliated with $\cR_\varphi$, then,
for $1\leq p<\infty$,
\[
  L_p(\cM)
  :=\{x\in L_0(\cR_\varphi,\tau_\varphi):
  \theta_s(x)=e^{-s/p}x\text{ for every }s\in\mathbb R\}.
\]
The canonical identification $L_1(\cM)\cong\cM_*$ determines a functional
$\operatorname{Tr}$ on $L_1(\cM)$, and
\[
  \|x\|_p
  =\bigl(\operatorname{Tr}(|x|^p)\bigr)^{1/p},
  \qquad x\in L_p(\cM).
\]
Up to canonical isometry, this construction is independent of $\varphi$;
in particular, it requires neither semifiniteness of $\cM$ nor
separability \cite{HaagerupLp,Terp,JungeXuReduction}.

\section{Proof of the main result}

\begin{lemma}[{\cite[Lemma 2.1]{Ding1}}]\label{lem:antipodes}
Let $E$ and $F$ be normed spaces, with $E$ strictly convex, and let
$V_0:S(E)\to S(F)$ be a mapping. Suppose that
\[
  -V_0(S(E))\subseteq V_0(S(E))
\]
and
\[
  \|V_0(x_1)-V_0(x_2)\|
  \leq \|x_1-x_2\|,
  \qquad x_1,x_2\in S(E).
\]
Then $V_0$ is one-to-one and
\[
  V_0(-x)=-V_0(x),
  \qquad x\in S(E).
\]
\end{lemma}

\begin{remark}\label{rem:antipodes}
The space $L_p(\cM)$ is strictly convex for $1<p<\infty$; see
\cite[Theorem 4.2]{kosaki}. Consequently, Lemma \ref{lem:antipodes}
implies that every surjective isometry
\(
  V_0:S(L_p(\cM_1))\longrightarrow S(L_p(\cM_2))
\)
satisfies
\[
  V_0(-x)=-V_0(x),
  \qquad x\in S(L_p(\cM_1)).
\]
\end{remark}

\begin{lemma}[{\cite[Theorem~5.1]{RX}}]\label{lem:Clarkson}
Let $\cM$ be a von Neumann algebra and let $A,B\in L_p(\cM)$, where
$1\leq p<\infty$ and $p\ne2$. Equality
\[
  \|A+B\|_p^p+\|A-B\|_p^p
  =2(\|A\|_p^p+\|B\|_p^p)
\]
holds in the Clarkson inequality
\begin{align*}
  \|A+B\|_p^p+\|A-B\|_p^p
    &\leq 2(\|A\|_p^p+\|B\|_p^p), &&1\leq p<2,\\
  \|A+B\|_p^p+\|A-B\|_p^p
    &\geq 2(\|A\|_p^p+\|B\|_p^p), &&2<p<\infty,
\end{align*}
if and only if $A\perp B$.
\end{lemma}

\begin{lemma}\label{lem:orth-preserving}
Let $\cM_1$ and $\cM_2$ be von Neumann algebras, and let
\(
  V_0:S(L_p(\cM_1))\longrightarrow S(L_p(\cM_2))
\)
be a surjective isometry, where $1<p\ne 2<\infty$. Then
\[
  x\perp y
  \quad\Longleftrightarrow\quad
  V_0(x)\perp V_0(y)
\]
for all $x,y\in S(L_p(\cM_1))$.
\end{lemma}

\begin{proof}
Let $x,y\in S(L_p(\cM_1))$ satisfy $x\perp y$. Then
$$\left\|x+y\right\|_p^p=\left\|x-y\right\|_p^p=\left\|x\right\|_p^p+\left\|y\right\|_p^p.$$ 
Since $V_0$ is an isometry and by Remark \ref{rem:antipodes}, we have that
\begin{align*}
  \|V_0(x)+V_0(y)\|_p
    &=\|V_0(x)-V_0(-y)\|_p=\|x+y\|_p,\\
  \|V_0(x)-V_0(y)\|_p
    &=\|x-y\|_p.
\end{align*}
Thus, by Lemma \ref{lem:Clarkson}, we have $V_0(x)\perp V_0(y)$. The converse follows by applying the same argument to \(V_0^{-1}\).
\end{proof}

\begin{theorem}\label{thm:orthogonal-addition}
Let $\cM_1$ and $\cM_2$ be von Neumann algebras, and let
\[
  V_0:S(L_p(\cM_1))
  \longrightarrow S(L_p(\cM_2))
\]
be a surjective isometry, where $1<p<\infty$ and $p\ne2$. If
$x,y\in S(L_p(\cM_1))$ satisfy $x\perp y$, then
\[
  V_0(ax+by)=aV_0(x)+bV_0(y)
\]
for every $a,b\in\mathbb R$ satisfying $|a|^p+|b|^p=1$.
\end{theorem}

\begin{proof}
By Remark \ref{lem:antipodes}, we only need to prove when \(a,b>0\).

Put
\[
        h_+=ax+by,\qquad h_-=ax-by,
\]
and set
\[
        z=V_0(h_+),\qquad w=V_0(h_-),\qquad
        X=V_0(x),\qquad Y=V_0(y).
\]
By Remark \ref{rem:antipodes}, we have that
\begin{align}\label{-h}
V_0(-h_-)=-w.
\end{align}
Then we observe that
\begin{align}\label{z+w}
       \|z+w\|_p \stackrel{\eqref{-h}}{=}\|h_++h_-\|_p=2a,\qquad
        \|z-w\|_p=\|h_+-h_-\|_p=2b.
\end{align}
Let
\[
        z_1=\frac{z+w}{2},\qquad z_2=\frac{z-w}{2}.
\]
Then 
\begin{align}\label{z_1,z_2}
\|z_1\|_p\stackrel{\eqref{z+w}}{=}a, \quad \|z_2\|_p\stackrel{\eqref{z+w}}{=}b
\end{align}
and 
\begin{align}\label{z_1+z_2}
z=z_1+z_2,\quad w=z_1-z_2.
\end{align}

{\bf Case 1: \(p>2\).}  
Since
\[
        \|z-X\|_p^p=\|w-X\|_p^p=(1-a)^p+b^p,
\]
it follows from Lemma \ref{lem:Clarkson} with \(A=z_1-X\) and \(B=z_2\) that
\begin{align*}
        (1-a)^p+b^p&=
        \frac{\|z-X\|_p^p+\|w-X\|_p^p}{2}\stackrel{\eqref{z_1+z_2}}{=}
        \frac{\|z_1-X+z_2\|_p^p+\|z_1-X-z_2\|_p^p}{2}
        \\& \ge \|z_1-X\|_p^p+b^p.
\end{align*}
Thus 
\begin{align}\label{z_1-x le}
\|z_1-X\|_p\le 1-a.
\end{align}
By triangle inequality, we have that
\begin{align}\label{z_1-x ge}
\|z_1-X\|_p\stackrel{\eqref{z_1,z_2}}{\ge} 1-a.
\end{align}
Hence
\[
        \|z_1-X\|_p\stackrel{\eqref{z_1-x le},\eqref{z_1-x ge}}{=}1-a.
\]
By strict convexity of \(L_p(\cM_2)\), we have
\begin{align}\label{z_1=aX}
    z_1=aX.
\end{align}

Similarly, 
applying Lemma \ref{lem:Clarkson}  with \(A=z_2-Y\), \(B=z_1\), and using
\[
        \|z-Y\|_p^p=\|w+Y\|_p^p=a^p+(1-b)^p,
\]
we obtain that 
\begin{align}\label{z_2=bY}
    z_2=bY.
\end{align}
Hence
\[
        z\stackrel{\eqref{z_1+z_2}}{=}z_1+z_2\stackrel{\eqref{z_1=aX},\eqref{z_2=bY}}{=}aX+bY.
\]

{\bf Case 2: \(1<p<2\).}  We can apply Lemma \ref{lem:Clarkson}  with \(A=z_1+X\), \(B=z_2\), and then with \(A=z_2+Y\), \(B=z_1\).
The same triangle equality argument gives \(z_1=aX\) and \(z_2=bY\).

Therefore, $V_0(ax+by)=aV_0(x)+bV_0(y)$ in all cases.
\end{proof}

For a surjective sphere isometry $V_0$ as above, its \emph{radial extension}
is the map defined by
\begin{align}\label{def V}
  \widetilde V(0):=0,
  \qquad
  \widetilde V(x)
  :=\|x\|_pV_0\left(\frac{x}{\|x\|_p}\right)
  \quad(x\ne0).
\end{align}

\begin{theorem}
Let $\cM_1$ and $\cM_2$ be finite-dimensional von Neumann algebras, and let
\(
  V_0:S(L_p(\cM_1))\longrightarrow S(L_p(\cM_2))
\)
be a surjective isometry, where $1<p\ne 2<\infty$. Then $\cM_2$ is
finite-dimensional, and the radial extension
\[
  \widetilde V:L_p(\cM_1)\longrightarrow L_p(\cM_2)
\]
is a surjective real-linear isometry.
\end{theorem}

\begin{proof}
We observe that $S(L_p(\cM_1))$ is compact. Since $V_0$ is a surjective isometry, the sphere of
$S(L_p(\cM_2))$ is compact, also. Hence, $L_p(\cM_2)$, and therefore $\cM_2$, is
finite-dimensional. 

Choose faithful normal traces $\tau_i$ on $\cM_i$,
$i=1,2$, and use the standard isometric identifications of the Haagerup
spaces $L_p(\cM_i)$ with the tracial spaces $L_p(\cM_i,\tau_i)$; see
\cite{kosaki,RX}. Since $\cM_i$ is finite-dimensional,
$\tau_i(I_i)<\infty$, and
\[
  \|x\|_{p,\tau_i}^p
  =\tau_i(|x|^p)
  \leq \|x\|_{\cM_i}^p\tau_i(I_i)
  <\infty
  \qquad(x\in\cM_i).
\]
Moreover, $L_p(\cM_i,\tau_i)$ is the completion of $\cM_i$ in the
$L_p$-norm, and the finite-dimensional space $\cM_i$ is already complete
for this norm. Consequently,
\begin{equation}\label{eq:finite-Lp-realization}
  L_p(\cM_i)
  \cong L_p(\cM_i,\tau_i)
  =\cM_i
  \qquad(i=1,2).
\end{equation}

We observe that the radial extension of $V_0^{-1}$ is the inverse of $\widetilde V$. Then we have that \(\widetilde V\) is a bijection. 

Let $z_1,\ldots,z_m$ be the central atoms of $\cM_1$. For
$1\leq\alpha\leq m$, put
\begin{align}\label{def:alpha}
  a_\alpha
  &:=
  \bigvee_{0\ne x\in z_\alpha L_p(\cM_1)}
  \ell(\widetilde V(x)),
  &
  b_\alpha
  &:=
  \bigvee_{0\ne x\in z_\alpha L_p(\cM_1)}
  r(\widetilde V(x)).
\end{align}
Since the central bands $z_\alpha L_p(\cM_1)$ are mutually orthogonal, it follows from Lemma \ref{lem:orth-preserving} that
\begin{align}\label{a_alpha perp a_beta}
  a_\alpha\perp a_\beta,
  \qquad
  b_\alpha\perp b_\beta
  \qquad(\alpha\ne\beta).
\end{align}
Then it follows from the
surjectivity of $\widetilde V$ that
\begin{align}\label{sum alpha=I}
  \sum_{\alpha=1}^m a_\alpha=I_2,
  \qquad
  \sum_{\alpha=1}^m b_\alpha=I_2.
\end{align}
For every $x\in L_p(\cM_1)$, Theorem \ref{thm:orthogonal-addition} yields that
\begin{equation}\label{eq:finite-central-orthogonal-decomposition}
  \widetilde V(x)
  =\sum_{\alpha=1}^m\widetilde V(z_\alpha x).
\end{equation}

Fix $y\in L_p(\cM_2)$ and choose $x$ with $\widetilde V(x)=y$. Let
\[
  y_\alpha:=\widetilde V(z_\alpha x),
\]
then by \eqref{eq:finite-central-orthogonal-decomposition}, we have that
\begin{align}\label{y=sum}
  y=\sum_{\alpha=1}^m y_\alpha,
  \qquad
  y_\alpha\in a_\alpha L_p(\cM_2)b_\alpha,
\end{align}
which implies that
\begin{align}\label{y_gamma}
    y_\gamma=a_\gamma y_\gamma b_\gamma
  \qquad(1\leq\gamma\leq m). 
\end{align}
Then, \eqref{a_alpha perp a_beta}
gives
\begin{align*}
  a_\alpha yb_\alpha
  &=\sum_{\gamma=1}^m
    a_\alpha a_\gamma y_\gamma b_\gamma b_\alpha
    =y_\alpha,\\
  a_\alpha yb_\beta
  &=\sum_{\gamma=1}^m
    a_\alpha a_\gamma y_\gamma b_\gamma b_\beta
    =0
    \qquad(\alpha\ne\beta).
\end{align*}
Since $y$ was arbitrary, it follows from \eqref{eq:finite-Lp-realization} that
\[
  a_\alpha\cM_2b_\beta=\{0\}
  \qquad(\alpha\ne\beta).
\]
We observe that
\[
  a_\alpha
  \stackrel{\eqref{sum alpha=I}
  }{=}a_\alpha\sum_{\beta=1}^m b_\beta
  \stackrel{\eqref{a_alpha perp a_beta}}{=}a_\alpha b_\alpha
  \stackrel{\eqref{sum alpha=I}
  }{=}\sum_{\beta=1}^m a_\beta b_\alpha
  \stackrel{\eqref{a_alpha perp a_beta}}{=}b_\alpha.
\]
Set
\[
  w_\alpha:=a_\alpha=b_\alpha.
\]
Then
\begin{align}\label{w_alpha}
  w_\alpha\cM_2w_\beta=\{0\}
  \qquad(\alpha\ne\beta),
  \qquad
  \sum_{\alpha=1}^m w_\alpha=I_2.
\end{align}
We observe that, for $c\in\cM_2$, we have that
\begin{align*}
  w_\alpha c
  &\stackrel{\eqref{w_alpha}}{=}\sum_{\beta=1}^m w_\alpha c w_\beta
    \stackrel{\eqref{w_alpha}}{=}w_\alpha c w_\alpha\stackrel{\eqref{w_alpha}}{=}\sum_{\beta=1}^m w_\beta c w_\alpha=
  cw_\alpha,
\end{align*}
which implies that
$w_\alpha\in\mathcal Z(\cM_2)$.

\medskip
\noindent\textbf{Claim.} For every $1\leq\alpha\leq m$,
\begin{equation}\label{eq:finite-central-band-image}
  \widetilde V\bigl(z_\alpha L_p(\cM_1)\bigr)
  =w_\alpha L_p(\cM_2).
\end{equation}

\noindent\emph{Proof of the claim.}
First, let $x\in z_\alpha L_p(\cM_1)$. If $x\ne0$, it follows from \eqref{def:alpha} that
\[
  \ell(\widetilde V(x))\leq a_\alpha=w_\alpha,
  \qquad
  r(\widetilde V(x))\leq b_\alpha=w_\alpha.
\]
The same conclusion is immediate when $x=0$. Hence
\[
  \widetilde V(x)
  =w_\alpha\widetilde V(x)w_\alpha
  \in w_\alpha L_p(\cM_2),
\]
which implies that
\[
  \widetilde V\bigl(z_\alpha L_p(\cM_1)\bigr)
  \subseteq w_\alpha L_p(\cM_2).
\]

For the reverse inclusion, let $y\in w_\alpha L_p(\cM_2)$. By surjectivity of $\widetilde V$, we can choose
$x\in L_p(\cM_1)$ such that $\widetilde V(x)=y$, and put
\[
  y_\beta:=\widetilde V(z_\beta x).
\]
Combining \eqref{y=sum} and \eqref{y_gamma}, we have that
\[
  y=\sum_{\beta=1}^m y_\beta,
  \qquad
  y_\beta=w_\beta yw_\beta.
\]
Since $w_\alpha$ is central and $y\in w_\alpha L_p(\cM_2)$,
\[
  y=w_\alpha yw_\alpha.
\]
Consequently, for $\beta\ne\alpha$,
\begin{align}\label{z_beta x=0}
  \widetilde V(z_\beta x)
  =w_\beta w_\alpha y w_\alpha w_\beta
  =0.
\end{align}
By \eqref{def V}, We observe that
$$\|\widetilde V(x)\|_p=\|x\|_p,$$
combining with \eqref{z_beta x=0}, we have that
$$z_\beta x=0,\qquad \beta\ne\alpha,$$ 
which implies that
$$x=z_\alpha x.$$ 
Hence
\[
  y=\widetilde V(x)
  \in\widetilde V\bigl(z_\alpha L_p(\cM_1)\bigr),
\]
which implies that
\[
  w_\alpha L_p(\cM_2)
  \subseteq\widetilde V\bigl(z_\alpha L_p(\cM_1)\bigr).
\]
This proves the claim.

We next show directly that every $w_\alpha$ is a central atom. Suppose, to
the contrary, that
\[
  w_\alpha=c_1+c_2,
\]
where $c_1,c_2$ are nonzero orthogonal central projections of $\cM_2$.
By \eqref{eq:finite-central-band-image}, the restriction
\[
  V_{0,\alpha}:=
  V_0|_{S(z_\alpha L_p(\cM_1))}:
  S(z_\alpha L_p(\cM_1))\longrightarrow
  S(w_\alpha L_p(\cM_2))
\]
is a surjective isometry. Applying the same argument of the central projections $w_\alpha$ to $V_0^{-1}$, we have that
\[
  z_\alpha=r_1+r_2,
\]
where $r_1,r_2$ are nonzero orthogonal central projections of $\cM_1$, contradicting to the atomicity of $z_\alpha$.
Hence, $w_\alpha$ is a central atom of $\cM_2$.

Since $z_\alpha$ and $w_\alpha$ are central atoms, the spaces
$z_\alpha L_p(\cM_1)$ and $w_\alpha L_p(\cM_2)$ are linearly isometric to
finite-dimensional Schatten $p$-classes; see
\cite[Section~1.14 and Remark~6.2.3]{DPS}.

By \eqref{eq:finite-central-band-image}, the restriction
\[
  V_\alpha
  :=V_0|_{S(z_\alpha L_p(\cM_1))}:
  S(z_\alpha L_p(\cM_1))
  \longrightarrow S(w_\alpha L_p(\cM_2))
\]
is therefore a surjective isometry between the unit spheres of two Schatten
$p$-classes. By \cite[Theorem~1.1]{HuangZhu}, it extends to a surjective
real-linear isometry
\[
  T_\alpha:z_\alpha L_p(\cM_1)
  \longrightarrow w_\alpha L_p(\cM_2)
\]
whose restriction to the unit sphere is $V_\alpha$.

Finally, the elements $z_\alpha x$ are mutually orthogonal. Hence, \eqref{eq:finite-central-orthogonal-decomposition} yields that
\[
  \widetilde V(x)
  =\sum_{\alpha=1}^m\widetilde V(z_\alpha x)
  =\sum_{\alpha=1}^mT_\alpha(z_\alpha x).
\]
which implies that $\widetilde V$ is a surjective real-linear isometry.
\end{proof}

\begin{theorem}
Let $\cM_1$ and $\cM_2$ be infinite-dimensional von Neumann algebras, let
\(
  V_0:S(L_p(\cM_1))\longrightarrow S(L_p(\cM_2))
\)
be a surjective isometry, where $1<p\ne 2<\infty$, and let
$\widetilde V$ be its radial extension. Then
\[
  \widetilde V:L_p(\cM_1)\longrightarrow L_p(\cM_2)
\]
is a surjective real-linear isometry.
\end{theorem}

\begin{proof}
We observe that the radial extension of $V_0^{-1}$ is the inverse of $\widetilde V$. Then we have that \(\widetilde V\) is a bijection.
Moreover, it follows directly from \eqref{def V} that $\widetilde V$ is continuous.

For $x,y\in L_p(\cM_1)$, we  put
\[
  e:=\ell(x)\vee r(x)\vee\ell(y)\vee r(y).
\]
We will prove that \begin{equation}\label{eq:radial-distance-with-common-orthogonal-vector}
  \|\widetilde V(x)-\widetilde V(y)\|_p=\|x-y\|_p.
\end{equation}
The \eqref{eq:radial-distance-with-common-orthogonal-vector} holds immediately when $x=0$ or $y=0$. Then we assume that $x,y\ne 0$.

We first prove \eqref{eq:radial-distance-with-common-orthogonal-vector} if 
\begin{align}\label{e ne I_1}
    e\ne I_1,
\end{align}
By \eqref{e ne I_1}, there exists 
$0\ne z\perp x,y$ with $\|z\|_p=1$.
Choose
\[
  R>\max\{\|x\|_p,\|y\|_p\}
\]
and set
\[
  c_x:=\bigl(R^p-\|x\|_p^p\bigr)^{1/p},
  \qquad
  c_y:=\bigl(R^p-\|y\|_p^p\bigr)^{1/p}.
\]
Then we have that
\[
  \left\|\frac{x+c_xz}{R}\right\|_p
  =
  \left\|\frac{y+c_yz}{R}\right\|_p
  =1.
\]
Applying Theorem \ref{thm:orthogonal-addition}  to the orthogonal
pairs $x,z$ after normalization, we have that
\begin{align}\label{V_0(x+c_xz}
  R V_0\left(\frac{x+c_xz}{R}\right)
  &=R\bigr(\frac{\|x\|}{R}V_0(\frac{x}{\|x\|})+\frac{c_x}{R}V_0(z)\bigl)
  \\&=
  \widetilde V(x)+c_xV_0(z).\nonumber
\end{align}
Similarly, we have that
\begin{align}\label{V_0(y+c_yz}
    R V_0\left(\frac{y+c_yz}{R}\right)
  =\widetilde V(y)+c_yV_0(z).
\end{align}
Since $z\perp x-y$, Lemma \ref{lem:orth-preserving} gives that
\begin{align}\label{z perp x-y}
    V_0(z)\perp \widetilde V(x)-\widetilde V(y).
\end{align}
 Then we have that
\begin{align*}\label{eq:common-z-distance-identity}
\|x-y\|_p^p+\|(c_x-c_y)z\|_p^p&= \|x-y+(c_x-c_y)z\|_p
  \\&=R^p\|\frac{x+c_xz}{R}-\frac{y+c_yz}{R}\|_p^p
  \\&=R^p\|V_0(\frac{x+c_xz}{R})-V_0(\frac{y+c_yz}{R})\|_p^p
  \\&\stackrel{\eqref{V_0(x+c_xz},\eqref{V_0(y+c_yz}}{=}
  \|\widetilde V(x)-\widetilde V(y)+(c_x-c_y)V_0(z)\|_p^p
  \\&\stackrel{\eqref{z perp x-y}}{=}\|\widetilde V(x)-\widetilde V(y)\|_p^p+\|(c_x-c_y)V_0(z)\|_p^p.
\end{align*}
which implies that
$$\|\widetilde V(x)-\widetilde V(y)\|_p=\|x-y\|_p.$$

Then for $e=I_1$. Suppose that $e=I_1$. Every infinite-dimensional von Neumann algebra
contains a sequence $(q_n)$ of mutually orthogonal nonzero projections
\cite[III.1.1.17]{Blackadar}. If $q=\bigvee_{n=1}^{\infty}q_n$, then, for
every vector $\xi$ in the underlying Hilbert space,
\[
  \sum_{n=1}^{\infty}\|q_n\xi\|^2
  =\langle q\xi,\xi\rangle
  \leq\|\xi\|^2.
\]
Hence $\|q_n\xi\|\to0$ for every $\xi$, and therefore 
\[
  q_n\xrightarrow{\mathrm{SOT}}0.
\]
Put
\[
  e_n:=I_1-q_n.
\]
Then $e_n<I_1$ and $e_n\to I_1$ strongly. The bimodule continuity of
Haagerup $L_p$-spaces gives (see, for example,
\cite[Section~2]{RX})
\[
  x_n:=e_nxe_n\longrightarrow x,
  \qquad
  y_n:=e_nye_n\longrightarrow y
  \quad\text{in }L_p\text{-norm}.
\]
For every $n$, choose
\[
  0\ne z_n\in q_nL_p(\cM_1)q_n.
\]
Then $z_n\perp x_n,y_n$, so
\eqref{eq:radial-distance-with-common-orthogonal-vector} yields
\[
  \|\widetilde V(x_n)-\widetilde V(y_n)\|_p
  =\|x_n-y_n\|_p.
\]
Letting $n\to\infty$ and using the continuity of $\widetilde V$, we obtain
$$\|\widetilde V(x)-\widetilde V(y)\|_p=\|x-y\|_p,$$
for arbitrary \(x,y\in L_p(\cM)\).
Thus $\widetilde V$ is an
isometry on all of $L_p(\cM_1)$. Since \(\widetilde V\) is surjective and $\widetilde V(0)=0$,
the Mazur--Ulam theorem \cite{MazurUlam} now implies that $\widetilde V$ is
real linear.
\end{proof}

\bibliographystyle{amsalpha}

\end{document}